\documentclass[oneside]{amsart}
\usepackage[letterpaper, body={14.6cm, 22.5cm}, mag=1000]{geometry}
\usepackage{amsmath, amsthm, amssymb, enumerate}
\usepackage{amsfonts}

\usepackage[colorlinks=true,linkcolor=red, citecolor=red]{hyperref}
\title{Note on Caranti's Method of Construction of Miller groups}
\author{Rahul Dattatraya Kitture, Manoj K. Yadav}
\begin{document}
\maketitle
\begin{abstract}
The non-abelian groups with abelian group of automorphisms are widely studied. Following Earnley, such
groups are called  Miller groups, since the first example of such a group was given by
Miller in 1913. 
Many other examples of Miller $p$-groups have been constructed by several authors. 
Recently, A. Caranti [{\it Israel J. Mathematics {\bf 205} (2015), 235-246}]
provided module theoretic methods for constructing  non-special Miller $p$-groups from special Miller $p$-groups.
By constructing examples, we show that these methods do not always work.    
We  also provide a sufficient condition on special Miller $p$-group for which the methods of Caranti work.
\end{abstract}

\newtheorem{lemma}{Lemma}[section]
\newtheorem{definition}[lemma]{Definition}
\newtheorem{theorem}[lemma]{Theorem}
\newtheorem{example}[lemma]{Example}
\newtheorem{remark}[lemma]{Remark}
\newcommand{\overbar}[1]{\mkern 1.5mu\overline{\mkern-1.5mu#1\mkern-1.5mu}\mkern 1.5mu}
\section{Introduction}\label{sectin1}
In 1908, H. Hilton \cite[Appendix, Q.7]{hilton} posed the following question: 
{\it Can a non-Abelian group have an Abelian group of  automorphisms?}
G. A. Miller (1913) provided positive answer to the question. 
He constructed a non-abelian group of order $64$ 
whose automorphism group is elementary abelian of order $128$ (see \cite{miller}). 
Following \cite{earnley} we call such groups {\it Miller groups}.

After the example by Miller, 
the theory of Miller groups has been developed with various examples.
There are many necessary conditions for a $p$-group to be Miller 
(see \cite{earnley}, \cite{hopkins} and \cite{morigi2}), but no sufficient conditions.
Therefore it is difficult to construct examples of Miller groups. 

Several examples of Miller groups have been constructed by various approaches 
(see \cite{earnley}, \cite{heineken2}, \cite{yadav1}, \cite{yadav2}, \cite{jonah} and \cite{morigi1}). 
Many of these groups are special $p$-groups and are given by simple presentations, but the 
techniques used, to prove that these groups are Miller, are highly computational and considerably difficult.

Concerning these difficulties in the construction of Miller groups, 
A. Caranti \cite[\S 5, \S6]{caranti}, among other things,  provided two methods, termed here {\it Method 1} and {\it Method 2}, to 
construct non-special Miller $p$-group $G$ from a special Miller $p$-group $H$ (see \S 3 for a brief description). 
The methods are interesting as these involve simple and elegant module theoretic 
arguments, instead of cumbersome computations. But, unfortunately there remained a gap in the proof. 
In this paper we attempt to fill up this gap in one direction which was motivated by an observation from the first two theorems, Theorem \ref{theoremA} and \ref{theoremB} below. 

Before stating our main results, we recall a terminology from \cite{gorenstein}. 
According to the methods in \cite{caranti}, given a special Miller $p$-group $H$ and a cyclic $p$-group $\langle z\rangle$ of order $\geq p^2$, a group $G:=H\rtimes_M \langle z\rangle$ is constructed as {\it amalgamated  semi-direct product} of $H$ by $\langle z\rangle$ (amalgamated) over a subgroup $M\leq H'$ of order $p$ (see \cite[p.27]{gorenstein} or \S 2 for the  definition). 
With appropriate action of $z$ on $H$ and some conditions on $H$,  it was claimed in \cite{caranti} that for every choice of $M$ (of order $p$) in $H'$ but not in $H^p$, $G$ is a Miller group. We show that this is not always true. 
\vskip3mm
Let $H=\langle a,b,c,d\rangle$ be a $p$-group of class $2$ with the following additional relations:
\begin{equation*}
a^p=[a,c], \hskip3mm b^p=[a,bcd], \hskip3mm c^p=[b,cd],\hskip3mm d^p=[b,d].
\end{equation*}
The group $H$ is a special Miller $p$-group of order $p^{10}$ (see Lemma \ref{lemma3}). 
Note that $\langle [a,b]\rangle$ and $\langle [a,d]\rangle$ are subgroups of order $p$ in $H'$ but not in $H^p$. 

%
\newtheorem{thm}{Theorem}
\renewcommand\thethm{\Alph{thm}}
\begin{thm} \label{theoremA}
Let $H$ be the special Miller $p$-group as above, $\langle z\rangle$ a  cyclic group of order $p^2$ and 
$M\leq H'$ a subgroup of order $p$ with $M\nsubseteq H^p$.  Let $G_1=H\rtimes_M \langle z\rangle$, where $z$ acts trivially on $H$. 
Then $G_1$ is a non-special $p$-group and the following holds true:
\begin{enumerate}
 \item If $M=\langle [a,b]\rangle$, then $G_1$ is a Miller group.
 \item If $M=\langle [a,d]\rangle$, then $G_1$ is not a Miller group.
\end{enumerate}
\end{thm}

\begin{thm} \label{theoremB}
Let $H$ be the special Miller $p$-group as above, $\langle w\rangle$  a  cyclic group of order $p^3$ and 
$M\leq H'$ a subgroup of order $p$ with $M\nsubseteq H^p$. Let $G_2=H\rtimes_M \langle w\rangle$, 
where $w$ normalizes $H$ via the following non-inner central automorphism of $H$: 
\begin{equation*}
waw^{-1}= ad^p, \hskip3mm wbw^{-1}= b,\hskip3mm  wcw^{-1}= c,\hskip3mm wdw^{-1}=d.
\end{equation*}
Then $G_2$ is a non-special $p$-group of order $p^{12}$ and the following holds true:
\begin{enumerate}
 \item If $M=\langle [a,b]\rangle$, then $G_2$ is a Miller group.
 \item If $M=\langle [a,d]\rangle$, then $G_2$ is not a Miller group.
\end{enumerate}
\end{thm}
Note that, although part (2) of the above theorems provides counter-examples for the two methods, part (1) motivates to find a condition on the choice of $M$ which would imply that $G$ is a Miller group. 
We provide a sufficient condition on the choice of $M$ for which the methods work. This condition is stated in the following theorems, for which we set some common hypotheses.

{\bf (i)} Let $H$ be a special Miller $p$-group such that $H^p<H'$ and the map $H/H'\rightarrow H^p$, $hH'\mapsto h^p$ is injective. 

{\bf (ii)}  Let $M$ be a subgroup of order $p$ in $H'$ but not in $H^p$.
\begin{thm}\label{theoremC}
With $H$ and $M$ as in (i) and (ii) above, let $G_1=H\rtimes_M \langle z\rangle$, where $z$ acts trivially on $H$,  $o(z)=p^2$.  If $H/M$ is a special Miller $p$-group, then 
$G_1$ is a (non-special) Miller group.
\end{thm}
Let $H$ also satisfy the following condition: 

{\bf (iii)} $H'$ is freely generated by $[x_i,x_j]$, $1\leq i<j\leq n$, provided $\{ x_1,\ldots,x_n\}$ is a minimal generating set for $H$.
\begin{thm} \label{theoremD}
With $H$ and $M$ as in (i) to (iii) above, let $G_2=H\rtimes_M \langle w\rangle$, where  $o(w)>p^2$ and $w$ acts on $H$ via a non-inner central automorphism of $H$. 
If $H/M$ is a special Miller $p$-group, then $G_2$ is (non-special) Miller group.
\end{thm}

\section{Notations and Preliminaries}\label{section2}
We start by setting some notations for multiplicatively written finite $p$-group $G$.
By $Z(G)$, $\Phi(G)$ and $G'=\gamma_2(G)$, we denote, respectively,  the center, the
Frattini subgroup and the commutator subgroup of $G$. 
The index of a subgroup $H$ in $G$ is denoted by $[G\colon H]$. We write $[a,b]=aba^{-1}b^{-1}$. 
The automorphism group of $G$ is denoted by $\mathrm{Aut}(G)$.
For a $p$-group $G$, we set $\Omega_n(G)=\langle g\colon g^{p^n}=1 \rangle$ and $G^p=\langle
g^p\colon g\in G\rangle$.  
Throughout the paper, {\it $p$ will denote an odd prime}. 
We denote by $\mathbb{F}_p$, the finite field of order $p$. 
We write the maps on the left. 
%
%
%

An automorphism $\varphi$ of a group $G$ is said to be {\it a central automorphism} 
if it induces identity automorphism on $G/Z(G)$, 
or equivalently if it commutes with every inner automorphism of $G$.
Let $\mbox{Autcent}(G)$ denote the group of all the central automorphisms of $G$. 
Note that if $\mbox{Aut}(G)$ is abelian, then $\mbox{Autcent}(G)=\mbox{Aut}(G)$, i.e. every automorphism of $G$ is central. 

The following lemma will be frequently used in the paper without reference.
\begin{lemma}
If $G$ is a $p$-group of nilpotency class $2$ then for all $a,b,c\in G$, 
$$ [ab,c] = [a,c][b,c] \hskip3mm \mbox{ and } \hskip3mm [a,bc]=[a,b][a,c].$$
\end{lemma}

We recall a generalization of semi-direct product according to \cite[p. 28]{gorenstein}.

%
%
\begin{definition}
A group $G$ is said to be an amalgamated internal semi-direct product of subgroups $H$ by $K$ over $M$, written $G=H\rtimes_M K$, 
if $H\trianglelefteq G$, $G=HK$ and $H\cap K=M$. In particular, if $[H,K]=1$ then $M\subseteq Z(G)$, and we call $G$ the central product of $H$ by $K$ over $M$, written $G=H\times_M K$.
\end{definition}

\begin{remark}\rm
Let $H,K$ be groups, $M\leq H$, $\psi\colon K\rightarrow \mathrm{Aut}(H)$ a homomorphism, and $\theta\colon M\rightarrow K$ an injective homomorphism with 
$\theta(M)\trianglelefteq K$. Assume that $\theta$ and $\psi$ satisfy the following compatibility conditions:  for all $h\in H$, $k\in K$, and $m\in M$,
$$ h^{\psi(\theta(m))}=h^m \mbox{ and } \theta(m^{\psi(k)})=(\theta(m))^k,$$ 
where $x^y:=yxy^{-1}$. Then the quintuple $(H,K,M,\psi,\theta)$ defines a group $G$ of the form $\overbar{G}=\overbar{H}\,\,\overbar{K}$ where $\overbar{H}\trianglelefteq \overbar{G}$, $\overbar{H}\cap \overbar{K}=\overbar{M}\trianglelefteq \overbar{K}$ and $H,K,M$ are isomorphic to $\overbar{H},\overbar{K},\overbar{M}$ respectively, such that the action of $\overbar{K}$ on $\overbar{H}$ corresponds to the action of $K$ on $H$ via $\psi$ (see \cite[p. 27-28]{gorenstein}). We still write $G:=H\rtimes_M K$ and call  the {\it amalgamated external semi-direct product of $H$ by $K$ over $M$}.
  It is often convenient to make no distinction between {\it internal} and {\it external} amalgamated semi-direct products.
\end{remark}

The following result is well-known. 
\begin{lemma}\cite[Exercise 6, p.78]{zassenhaus}
\label{lem4}
Let $G$ be a finite group and $N$ a normal subgroup of $G$. Then the set of automorphisms of $G$ which are identity on $N$ as well as on $G/N$ form an abelian 
subgroup of $\mathrm{Aut}(G)$ whose exponent divides the exponent of $N$.
\end{lemma}
We recall some results on Miller groups.
\begin{theorem} \cite[\S 2, \S 3]{hopkins}
 \label{thm1}
If $H$ is a Miller $p$-group, then the following holds true:  

($i$) $H$ has no direct abelian factor. 

($ii$) Every automorphism of $H$ is identity on $H'$.
\end{theorem}
%
\begin{theorem}\cite[p.15]{earnley}
\label{thm2}
If $G$ is a Miller $p$-group then $Z(G)$ and $\Phi(G)$ are non-cyclic.
\end{theorem}
%
\begin{theorem}\cite[Main Theorem]{curran2}
\label{thm3}
If $G$ is a finite $p$-group, then $\mathrm{Autcent}(G)=\mathrm{Inn}(G)$ if and only if $Z(G)=G'$ and $Z(G)$ is cyclic.
\end{theorem}
%

Let $G=A\times B$ be a finite group. Let $i_A,i_B$ denote the natural injections from $A,B$ respectively into $G$ and $\pi_A,\pi_B$ the natural projections 
from $G$ onto $A,B$ respectively. For $\varphi\in \mathrm{Aut}(G)$, define 
$\alpha = \pi_A\varphi i_A$, $\beta= \pi_A\varphi i_B$, $\gamma=\pi_B\varphi i_A$ and $\delta = \pi_B \varphi i_B.$
Then for $a\in A$, $b\in B$ and $\varphi\in\mathrm{Aut}(A\times B)$,  we have $\varphi(a)=\alpha(a)\gamma(a)$ and 
$\varphi(b)=\beta(b)\delta(b)$. It is obvious that 
$$\alpha\in \mathrm{End}(A), \,\,\,\, \beta\in \mathrm{Hom}(B,A), \,\,\,\, \gamma\in \mathrm{Hom}(A,B), \,\,\,\, \mbox{ and} \,\,\,\,  \delta\in \mathrm{End}(B). $$
The following result follows from  \cite[Theorem 3.2]{bidwell}.
\begin{theorem}
\label{thm5}
With the above set up, suppose, $A$ and $B$ have no common direct factor. Then
 $\alpha\in \mathrm{Aut}(A)$,  $\beta\in \mathrm{Hom}(B,Z(A))$, $\gamma\in \mathrm{Hom}(A,Z(B))$ and  $\delta\in \mathrm{Aut}(B)$. 
\end{theorem}
%
%
%
%
\makeatletter
\def\subsection{\@startsection{subsection}{2}%
  \z@{.5\linespacing\@plus.7\linespacing}{.3\linespacing}%
  {\normalfont\bfseries}}
\makeatother
\vskip5mm
\section{Caranti's methods}\label{section3}
In this section, we briefly describe the methods of Caranti  in group theoretic set up. 
A $p$-group $H$ is said to be a {\it special} $p$-group if $Z(H)=H'=\Phi(H)$. 
Let $H$ be special $p$-group  with a presentation
\begin{equation}\label{eq1}
\begin{split}
H=\Big{\langle} x_1,x_2,\ldots, x_n \colon & [x_i,x_j,x_k]=1,\\
 & [x_i,x_j]^p=1,\\
 & x_i^p=\prod_{j<k} [x_j,x_k]^{c_{ijk}},\\
 & \prod_{j<k} [x_j,x_k]^{d_{ljk}}=1, l=1,2,\ldots,t\Big{\rangle},
\end{split}
\tag{*}
\end{equation}
where $c_{ijk}, d_{ljk}\in\mathbb{Z}$. Further, we have a well defined map 
\begin{equation*}
f\colon H/H'\rightarrow H^p, \hskip5mm f(xH')=x^p.
\end{equation*}
The quotient group $H/H'$ can be viewed as a vector space over $\mathbb{F}_p$. 
Every $\varphi$ $\in$ $\mathrm{Aut}(H)$ induces an automorphism $\alpha$ on $H/H'$. 
On the other hand, suppose that $\alpha$ is an automorphism of $H/H'$. 
The action of $\alpha$ on $H/H'=H/Z(H)$ completely determines its action on $H'$ by
\begin{equation*}
\hat{\alpha} \colon H'\rightarrow H', \hskip5mm \hat{\alpha}([x,y])=[\alpha(xH'),\alpha(yH')],
\end{equation*}
and if $\alpha\in \mbox{Aut}(H/H')$ is induced by an automorphism $\varphi$ of $H$, 
then the action of $\hat{\alpha}$ on $H'$ coincides with the restriction of $\varphi$ to $H'$. 

\begin{lemma}\label{lem6}
Let $H$ be a special $p$-group with presentation as in {\rm(\ref{eq1})}. Then the following statements hold true.
\begin{enumerate}[(i)]
\item  An automorphism $\alpha$ of $H/H'$ is induced by an automorphism of $H$ if and only if $\hat{\alpha}\circ f=f\circ \alpha.$
\item  Consequently, $\mathrm{Aut}(H)=\mathrm{Autcent}(H)$ if and only if 
\begin{equation*} 
\{ \alpha\in \mathrm{Aut}(H/H')\colon  \hat{\alpha}\circ f=f\circ \alpha \}=\{I\}.
\end{equation*}
\end{enumerate}
\end{lemma}
\begin{proof}
See \cite[\S 3]{caranti}.
\end{proof}

We now briefly describe the methods of Caranti \cite[\S 5, \S 6]{caranti}.
\subsection*{Method 1: To construct Miller $p$-group $G$ with $G'=\Phi(G)<Z(G)$}

Let $H$ be a special Miller $p$-group satisfying the following conditions:
\vskip1mm\hskip5mm
{\bf (I)} The map $f\colon H/H'\rightarrow H^p$, $f(xH')= x^p$ is injective (so $|H^p|=|H/H'|$). 
\vskip1mm\hskip5mm
{\bf (II)} $H^p<H'$. 

Suppose that $H$ has presentation (\ref{eq1}). Let $M\leq H'$ be a subgroup of order $p$ with $M\nsubseteq H^p$. 
Let $G= H\times_M \langle z\rangle$, where $o(z)=p^2$. 
Then $G=\langle x_1,\ldots, x_n,z\rangle$ and  
$$G'=H', \hskip5mm \Phi(G)=\Phi(H), \hskip5mm Z(G)=\langle Z(H), z\rangle.$$
Therefore, $G'=\Phi(G)<Z(G)$, i.e. $G$ is non-special $p$-group. 
Finally, it is proved in \cite{caranti}, 
that every automorphism of $G$ is central, in the following way.
A given $\varphi \in \mbox{Aut}(G)$ induces an automorphism $\alpha$ on the vector space $V=G/\Phi(G)=\langle \bar{x}_1, \ldots, \bar{x}_n,\bar{z}\rangle $. 
Since $Z(G)$ and $\Phi(G)$ are characteristic subgroups of $G$, 
 $\alpha$ leaves $Z(G)/\Phi(G)=\langle \bar{z}\rangle$ invariant, 
and therefore induces an automorphism on $V/\langle \bar{z}\rangle=\langle \bar{x}_1,  \ldots, \bar{x}_n\rangle\cong H/\Phi(H)$, 
which we still denote by $\alpha$. 
Then it is concluded, without proof, in \cite[p. 243]{caranti} that $\alpha$ satisfies the condition 
$$\hat{\alpha}\circ f=f\circ \alpha.$$ 
Equivalently, by Lemma \ref{lem6}(i), 
the automorphism $\alpha$ of $H/H'$ is induced by an automorphism of $H$. 
We show, by an example in \S \ref{section5}, that this is not always true, and $G$ may not be a Miller group.
\subsection*{Method 2: To construct Miller $p$-group $G$ with $G'<\Phi(G)=Z(G)$}
\leavevmode
Let $H$ and $M$ be as in Method 1. Let $H$ satisfies the following additional condition:
\vskip1mm\hskip5mm
{\bf (III)} $H'$ is freely generated by $[x_i,x_j]$, $1\leq i< j\leq n$, 
provided $\{x_1$,  $\ldots$, $x_n\}$ is a {\it minimal} generating set for $H$.
\vskip2mm
Thus $|H|=p^{n+\binom{n}{2}}$.
Since $H$ is a special Miller group, by Theorems \ref{thm2} and \ref{thm3}, 
$H$ possesses a non-inner central automorphism, say $\gamma$. 
(In [\ref{caranti}], the central automorphism $\gamma$ is chosen in such a way that 
the induced map $H/H'\rightarrow Z(H)$, $xH'\mapsto x^{-1}\gamma(x)$ is {\it injective} \cite[p. 244]{caranti},  only to ensure that $\gamma$ is 
{\it not an inner automorphism}.)
Let the presentation of $H$ be as in (\ref{eq1}).  Let $G=H\rtimes_M \langle z\rangle$, 
where $z$ normalizes $H$ in the following way.
\begin{equation*}
zx_iz^{-1}=\gamma(x_i), \hskip3mm z^{p^m}\in H'\setminus H^p, m\geq 2.
\end{equation*}
Note than $o(z)=p^{m+1}\geq p^3$. 
By Theorem \ref{thm1}(ii) and Lemma \ref{lem4}, Aut$(H)$ is elementary abelian. 
Then $z^px_iz^{-p}=\gamma^p(x_i)=x_i$ for all $i$, i.e. $z^p\in Z(G)$.
Thus, we have 
\begin{equation*}
Z(G)=\langle Z(H),z^p\rangle=\langle \Phi(H),z^p\rangle=\Phi(G).
\end{equation*}
Further, since $\gamma$ is a central automorphism, 
\begin{equation*}
[z,x_i]=zx_iz^{-1}x_i^{-1}=\gamma(x_i)x_i^{-1}\in Z(H)=H'.
\end{equation*}
Therefore $G'=H'$. 
Since $z^{p^m}\in H'$ with $m\geq 2$ and $z^p\notin H'=G'$,  $G'<\Phi(G)$.
Then it is proved in [\ref{caranti}] that $G$ is a Miller group, in the following way.
Let $\varphi \in \mbox{Aut}(G)$. Then $\varphi$ induces an automorphism on 
$G/Z(G)=G/\Phi(G)=\langle \bar{x}_1, \ldots, \bar{x}_n,
\bar{z}\rangle$. 
Since 
\begin{equation*}
\Omega_2(G)=\langle g\colon g^{p^2}=1\rangle = \langle x_1,\ldots, x_n,z^{p^{m-1}}\rangle
\end{equation*}
is characteristic, $\varphi$ leaves this subgroup invariant and induces an automorphism  $\alpha$ on the quotient
$\Omega_2(G)\Phi(G)/\Phi(G)=\langle \bar{x}_1, \ldots, \bar{x}_n\rangle\cong H/\Phi(H)$.
Then it is concluded, without proof, in \cite[p. 244]{caranti} that $\alpha$ satisfies the condition 
$$\hat{\alpha}\circ f=f\circ \alpha,$$
i.e. (again by Lemma \ref{lem6}(i)) the automorphism $\alpha$ of $H/H'$ is induced by an automorphism of $H$. 
Again we show, by an example in \S \ref{section5}, that this is not always true, and 
$G$ may not be a Miller group.
\section{A substitute for the methods of Caranti}\label{section4}
In this section, we fill up the gaps in the methods discussed in the preceding section by providing a {\it sufficient condition} on the 
choice of the subgroup $M$, which is used for the amalgamation in both the methods. 
%
%
\vskip5mm
\begin{proof}[Proof of Theorem \ref{theoremC}]
Let $H,M,\langle z\rangle$ and $G_1$ be as stated in the hypothesis of the theorem.
For simplicity, write $G_1=G$. We have $G=\langle H,z\rangle$. Note that $M=H\cap \langle z\rangle = \langle z^p\rangle$. 

Suppose that $\overbar{H}=H/M$ is a Miller group. 
Since $Z(G)=\langle Z(H),z\rangle$, $Z(G)^p=\langle z^p\rangle=M$ is characteristic subgroup of $G$.
\vskip2mm
{\it Claim 1: Every automorphism of $G/M$ is central.}
\vskip2mm
Let $\psi$ be an automorphism of $G/M$. Note that 
$$\frac{G}{M}=\frac{H}{M} \times \frac{\langle z\rangle}{M}=:\overbar{H} \times \langle \bar{z}\rangle.$$
Let $i_{\overbar{H}} $, $i_{\langle \bar{z}\rangle}$ denote the natural injections from $\overbar{H}$, $\langle \bar{z}\rangle$ respectively into 
$\overbar{H} \times \langle \bar{z}\rangle$, and $\pi_{\overbar{H}}$, $\pi_{\langle \bar{z}\rangle}$, the 
natural projections of $\overbar{H}\times \langle \bar{z}\rangle$ onto $\overbar{H}$,  $\langle \bar{z}\rangle$ respectively. 
The automorphism $\psi$ of $\overbar{H}\times \langle \bar{z}\rangle$ uniquely determines the four components 
\begin{align*}
& \hat{\alpha} = \pi_{\overbar{H}}\psi i_{\overbar{H}} \in \mathrm{End}(\overbar{H}) & \hat{\beta}= \pi_{\overbar{H}}\psi i_{\langle \bar{z}\rangle} \in \mathrm{Hom}(\langle \bar{z}\rangle,\overbar{H}),\\
& \hat{\gamma} = \pi_{\langle \bar{z}\rangle}\psi i_{\overbar{H}} \in\mathrm{Hom}(\overbar{H},\langle \bar{z}\rangle), & \hat{\delta} = \pi_{\langle \bar{z}\rangle}\psi i_{\langle \bar{z}\rangle} \mathrm{End}(\langle \bar{z}\rangle).
\end{align*}
Since $|\langle \bar{z}\rangle|=p$, by Theorem \ref{thm1}(i),  $\overbar{H}$ and $\langle \bar{z}\rangle$ have no common direct factor. 
By Theorem \ref{thm5}, $\hat{\alpha} \in \mathrm{Aut}(\overbar{H})$,  $\hat{\gamma}\in \mathrm{Hom}(\overbar{H}, \langle \bar{z}\rangle)$ and for any $x\in\overbar{H}$, $\psi(x)=\hat{\alpha}(x)\hat{\gamma}(x)$. 

Since $\overbar{H}$ is a Miller group, for any $x\in \overbar{H}$ we have $x^{-1} \hat{\alpha}(x)\in Z(\overbar{H})$. Thus
$$ x^{-1}\psi(x) = x^{-1}\hat{\alpha}(x) \hat{\gamma}(x) \in Z(\overbar{H})\times \langle \bar{z}\rangle=
Z(\overbar{H}\times \langle z\rangle), \hskip5mm \mbox{ for all } x\in\overbar{H}.$$
Since $\bar{z}$ is central in $\overbar{H}\times \langle \bar{z}\rangle$, the last equation 
implies that $\psi$ is a central automorphism of $\overbar{H}\times \langle \bar{z}\rangle$. This proves the claim.

\vskip2mm
{\it Claim 2: Every automorphism of $G$ is central.}
\vskip2mm

Both $H$ and $\overbar{H}=H/M$ are special $p$-groups and $M\subseteq \gamma_2(H)=Z(H)$, hence 
$Z(\overbar{H})=\gamma_2(\overbar{H}) = \gamma_2(H)/M=Z(H)/M$. Then 
\begin{align*}
Z(G/M) = Z(\overbar{H}\times \langle \bar{z}\rangle) = Z(\overbar{H})\times \langle \bar{z}\rangle = \frac{Z(H)}{M} \times \frac{\langle z\rangle}{M} = \frac{Z(H)\langle z\rangle}{M}=\frac{Z(G)}{M}.
\end{align*}
Since $M$ is characteristic in $G$, any $\varphi\in \mathrm{Aut}(G)$ induces an automorphism on $G/M$, which is central (by Claim 1). It follows that 
$\varphi$ is a central automorphism of $G$. 

\vskip2mm
{\it Claim 3: $\mathrm{Aut}(G)$ is abelian.}
\vskip2mm
Let $\varphi\in\mathrm{Aut}(G)$ and $\{ x_1,\ldots, x_n\}$ be a minimal generating set for $H$. Then 
$\{x_1,\ldots,x_n,z\}$ is a minimal generating set for $G$ and (as seen in Method 1, \S \ref{section3}) $\Phi(G)<Z(G)=\langle Z(H),z\rangle = \langle \Phi(G),z\rangle$. 
Now $G/\Phi(G)$ is a vector space with 
$\{\tilde{x}_1$, $\ldots$, $\tilde{x}_n$,$\tilde{z}\}$ as a basis, and $Z(G)/\Phi(G)=\langle \tilde{z}\rangle$ is a $\varphi$-invariant subspace. 
Thus, if 
$$ \varphi(\tilde{x}_j) = \sum_{i=1}^n \alpha_{ij} \tilde{x}_i + \lambda_j \tilde{z}_j \mbox{ and } \varphi(\tilde{z}) = \mu \tilde{z}, \hskip3mm (j=1,\ldots,n),$$
then the matrix of the action of $\varphi$ on $G/\Phi(G)$ is
$$\begin{bmatrix}
\alpha & 0\\
\lambda & \mu
\end{bmatrix}, \hskip5mm \alpha=[\alpha_{ij}]_{n\times n}, \lambda=[\lambda_1 \, \cdots \, \lambda_n].
$$
Here $\alpha$ is the matrix of the action of $\varphi$ on $G/Z(G)=\langle \tilde{x}_1, \ldots, \tilde{x}_n\rangle$ with basis
$\{\tilde{x}_1, \ldots, \tilde{x}_n \}$, which is, by Claim 2,  identity. Thus, $\alpha=I_{n\times n}$ in the above matrix. 
From here onwards, we can continue the arguments of \cite{caranti} (p. 243, after proving that $\alpha=1$) to conclude that $\mathrm{Aut}(G)=\mathrm{Aut}(G_1)$ is abelian.
\end{proof}
\vskip5mm
\begin{proof}[Proof of Theorem \ref{theoremD}]
We have $H, M, \langle w\rangle$ and $G_2=H\rtimes_M \langle w\rangle$ be as stated in the theorem.
We have $o(w)=p^{m+1}\geq p^3$, and $M=H\cap \langle w\rangle$, a subgroup of order $p$ in $H'$ but not in $H^p$. 
Then $w^{p^{m}}\in H'\setminus H^p$. Further (as seen in \S 3), $\Phi(G_2)=Z(G_2)=\langle Z(H),w^p\rangle$.

Set $z=w^{p^{m-1}}$ and $G_1=H\langle z\rangle$. Then $o(z)=p^2$, and $z$ centralizes $H$ (for, $m\geq 2$ so  
$\langle z\rangle =\langle w^{p^{m-1}}\rangle \subseteq \langle w^p\rangle\subseteq Z(G)$). Thus, $G_1$ satisfies the hypothesis in Theorem \ref{theoremC},  $G_1$ is a Miller group.
Now $G_1$ is characteristic subgroup of $G_2$, since 
\begin{equation*}
G_1 = \Omega_2(G_2)= \langle g\in G\colon g^{p^2}=1\rangle.
\end{equation*}
Let $\{ x_1,\ldots,x_n\}$, a minimal generating set for $H$. Then $\{ w,x_1,\ldots,x_n\}$ is a minimal generating set for $G_2$. 
Now $G_2/Z(G_2)$ is a vector space over $\mathbb{F}_p$ with $\{\bar{w}$, $\bar{x}_1$, $\ldots$, $\bar{x}_n\}$ as a basis. 
Thus, if $\varphi\in \mathrm{Aut}(G_2)$ then $\varphi$ acts on $G_2/Z(G_2)$ 
and the subspace $\Omega_2(G_2)Z(G_2)/Z(G_2)$ $=$ $\langle \bar{x}_1$, $\ldots$, $\bar{x}_n\rangle$ is $\varphi$-invariant. 
Thus, if 
\begin{equation*}
\varphi(\bar{w})=\tau \bar{w} + \sum_{j=1}^n \sigma_j \bar{x}_j \mbox{ and } \varphi(\bar{x}_i)=\sum_{j=1}^n \alpha_{ij}\bar{x}_i, \hskip5mm (i=1,2,\ldots,n),
\end{equation*}
then the matrix of the action of $\varphi$ on $G_2/Z(G_2)$ w.r.t. the above basis is
\begin{equation*}
\begin{bmatrix}
\tau & 0 \\
\sigma
& \alpha
\end{bmatrix},
\end{equation*}
where $\tau\in\mathbb{F}_p$, $\sigma=[\sigma_1 \, \cdots \sigma_n]^t$ and $\alpha=[\alpha_{ij}]_{n\times n}$. 
Recall that $G_1$ is a Miller group which is characteristic in $G_2$. 
Hence $\varphi$ is identity on $G_1/Z(G_1)$. In particular 
$\varphi(x_i)\equiv x_i \pmod{Z(G_1)}$. 
But, since $m\geq 2$, 
\begin{equation*}
Z(G_1)=\langle Z(H),z\rangle = \langle Z(H),w^{p^{m-2}}\rangle \leq \langle Z(H),w^p\rangle =Z(G_2);
\end{equation*}
it follows that $\varphi(x_i)\equiv x_i  \pmod{Z(G_2)}$, i.e. $\alpha=I_n$ in the above matrix. 
This conclusion now allows us to continue the argument of \cite{caranti} (p. 244, after proving that $\alpha=1$) to conclude 
that $G_2$ is a Miller group.

\end{proof}

\section{Proofs of Theorems A and B}\label{section5}
%
%
%
Before starting the proofs of Theorems \ref{theoremA} and \ref{theoremB}, we make a remark.
Since Miller $p$-groups are generated by at least $4$ elements (see \cite{morigi2}), we have $|H/H'|\geq p^4$. By conditions (I) and (II) in \S \ref{section3}, 
$|H'|>|H^p|\geq p^4$. Therefore, $|H|\geq p^9$. If $H$ also satisfies condition (III) in \S \ref{section3}, then $|H|=p^{n+\binom{n}{2}}\geq p^{4+\binom{4}{2}}=p^{10}$. 
There are examples of special Miller $p$-groups in the literature satisfying all these conditions (see \cite{heineken2}), but  these are of large order ($p^{45}$). 
We start with the following example of minimum order.

\begin{example}\label{example2}
{\rm Let $H=\langle a,b,c,d\rangle$ be a $p$-group of class $2$ with following additional relations:
\begin{equation*}
a^p=[a,c], \,\,\,\, b^p=[a,bcd], \,\,\,\, c^p=[b,cd], \,\,\,\, d^p=[b,d].
\end{equation*}
It is easy to see that $H$ is a special $p$-group of order $p^{10}$ with 
\begin{equation*}
Z(H)=\Phi(H)=H'=\langle [a,b], [a,c], [a,d], [b,c], [b,d],[c,d]\rangle.
\end{equation*}}
\end{example}
\begin{lemma}\label{lemma3}
The group $H$ in Example \ref{example2} is a Miller group.
\end{lemma}
\begin{proof}
For simplicity, write $Z(H)=Z$. Let $\overbar{H}=H/H^p$.
Then $\overbar{H}=\langle \bar{a}, \bar{b}, \bar{c}, \bar{d}\rangle$, 
and the following relations hold in $\overbar{H}$:
\begin{equation*}
[\bar{a},\bar{c}]=1, \,\,\,\, [\bar{a},\bar{b}\bar{c}\bar{d}]=1, \,\,\,\, [\bar{b},\bar{c}\bar{d}]=1,\,\,\,\, [\bar{b},\bar{d}]=1.
\end{equation*}
These relations can be rewritten in the following way:
\begin{equation*}
[\bar{a},\bar{c}]=[\bar{a},\bar{b}\bar{d}]=1 \mbox{ and } [\bar{b},\bar{c}]=[\bar{b},\bar{b}\bar{d}]=1.
\end{equation*}
Then these relations imply that 
$\overbar{H} = \langle \bar{a},\bar{b}\rangle \times \langle \bar{c},\bar{b}\bar{d} \rangle = H_1\times H_2$, the 
internal direct product of two non-abelian subgroups of order $p^3$. 
Any $h_i$ in $H_i$ has at most $p$ conjugates in $H_i$ ($i=1,2$). 
Hence an element $h_1h_2\in H_1\times H_2$ will have exactly $p$-conjugates 
if and only if exactly one of $h_i$ ($i=1,2$) is central in $H_i$. 
Thus, the elements in $\overbar{H}=H_1\times H_2$ with exactly $p$ conjugates constitute the sets 
\begin{align*}
S_1 &=\{ \bar{a}^i\bar{b}^jz_1\colon (i,j)\neq (0,0), \hskip2mm  i,j\in\mathbb{F}_p, \hskip2mm z_1\in Z(\overbar{H})\}\\
\mbox{ and } S_2&= \{ \bar{c}^k (\bar{b}\bar{d})^lz_2\colon (k,l)\neq (0,0),\hskip2mm k,l\in\mathbb{F}_p, \hskip2mm z_2\in Z(\overbar{H})\};
\end{align*}
hence $S_1\cup S_2$ is a characteristic {\it subset} of $\overbar{H}$. 
Given any $x_i\in S_i$, 
the only elements in $S_1\cup S_2$ which do not commute with $x_i$  lie within $S_i$, 
hence an automorphism of $\overbar{H}$ either interchanges $S_1$, $S_2$ 
or  leaves them invariant. 
Note that $\langle S_1\rangle=\langle \bar{a},\bar{b},Z(\overbar{H})\rangle$ 
and  $\langle S_2\rangle=\langle \bar{c},\bar{b}\bar{d},Z(\overbar{H})\rangle$. 
Consequently, in $H$, 
any automorphism either interchanges the subgroups $\langle a,b,Z\rangle$ 
and $\langle c,bd,Z\rangle$ or leaves them invariant. 

Let $A=\langle a,b,Z\rangle$, $B=\langle c,bd,Z\rangle$. 
Then $A^p=\langle a^p,b^p\rangle$, $B^p=\langle c^p,b^pd^p\rangle$. Next, there is $x\in A$ such that $A^p\leq [x,H]$. For
$$\langle a^p,b^p\rangle = \langle [a,c], [a,bcd]\rangle \leq [a,H].$$ 
But, it can be shown that, there is no $y\in B$ such that $B^p\leq [y,H]$. 
Hence $A$ and $B$ can not be interchanged by any automorphism of $H$, 
they must be characteristic subgroups of $H$. 

Now, the only elements $x\in A$ with $A^p\leq [x,H]$ are the elements 
$a^iz$, where $i\neq 0$ and $z\in Z$; these elements together with $Z$ generate the characteristic subgroup $\langle a,Z\rangle$ of $H$. Then 
$\langle a,Z\rangle^p=\langle a^p\rangle$ is characteristic in $H$, and one can see that for arbitrary $c^i(bd)^jz$ in the characteristic subgroup $B$, 
\begin{equation*}
\langle a^p\rangle \leq [c^i(bd)^jz,H] \mbox{ if and only if } j=0 \mbox{ and } i\neq 0.
\end{equation*}
Thus the elements $c^iz$, where $i\neq 0$ and $z\in H$, together with $Z$ generate the characteristic subgroup $\langle c,Z\rangle$ of $H$. 
Then $\langle c,Z\rangle^p=\langle c^p\rangle$ is characteristic in $H$. We can see that 
\begin{align*}
& \langle c^p\rangle \leq [a^ib^jz,H] \mbox{ if and only if } i=0 \mbox{ and } j\neq 0\\
\mbox{ and } & \langle c^p\rangle \leq [c^i(bd)^jz,H] \mbox{ if and only if } i=j\neq 0.
\end{align*}
We get, as before, that $\langle b,Z\rangle$ and $\langle cbd,Z\rangle$ are characteristic subgroups of $H$. 
Note that $cbdZ=bcdZ$, since $H/Z$ is abelian. We have obtained the following characteristic subgroups of $H$:
$$ \langle a,Z\rangle, \langle b,Z\rangle, \langle c,Z\rangle, \langle bcd, Z\rangle.$$
As a consequence, for any automorphism $\varphi$ of $H$, we have
$$ \varphi\colon a\mapsto a^iz_1, \,\,\, b\mapsto b^jz_2, \,\,\, c\mapsto c^kz_3, \,\,\, bcd\mapsto (bcd)^lz_4,$$
where $i,j,k,l\in\mathbb{F}_p^{\times}$ and $z_1,z_2,z_3,z_4\in Z$. Let 
$$\varphi(d)=a^rb^sc^td^uz', \hskip5mm (r,s,t,u\in\mathbb{F}_p, z'\in Z).$$

(i) Applying $\varphi$ to the relation $[a,c]=a^p$ gives $k=1$. 

(ii) Applying $\varphi$ to the relation $[a,bcd]=b^p$ gives $il=j$. 

(iii) Applying $\varphi$ to the relation $[b,cd]=c^p$ (i.e. $[b,bcd]=c^p$) gives $jl=1$. 

(iv) Apply $\varphi$ to the relation $[b,d]=d^p$ and express $p$-th powers as commutators: 
\begin{equation*}
[b^jz_2, a^rb^sc^td^uz'] = a^{rp} b^{sp} c^{tp} d^{up},
\end{equation*}
which on simplification gives
\begin{equation*}
[a,b]^{-jr} [b,c]^{jt} [b,d]^{ju} = [a,c]^r [a,b]^s [a,c]^s [a,d]^s [b,c]^t [b,d]^t [b,d]^u.
\end{equation*}
Comparing coefficients of $[a,c]$ and $[a,d]$, we get $r+s=0$, $s=0$;  hence $r=s=0$ and 
$$[b,c]^{jt} [b,d]^{ju} = [b,c]^t [b,d]^{t+u} \hskip1cm (**).$$
This implies that 
$$jt=t \mbox{ and } ju=t+u.$$
If $j\neq 1$ then $t=0$ and $ju=u$, hence $u=0$. But then $\varphi(d)=z'\in Z$, a contradiction (since $d\notin Z$). Hence $j=1$, and using this in (ii) and (iii) 
we get $i=l=1$. Thus $\varphi$ is identity modulo $Z$ on $\{a,b,c,bcd\}$. 
Since $a,b,c,bcd$ generate $H$, $\varphi$ is identity on $H/Z(H)$ and so is on $H'=Z(H)$ (by \cite[\S 2]{hopkins}).  
Thus by Lemma \ref{lem4}, $\mathrm{Aut}(G)$ is elementary abelian.
\end{proof}
We are now ready to prove Theorems \ref{theoremA} and \ref{theoremB}.
%
Let $H$ be the group as in Example \ref{example2}. 
Note that $[a,b], [a,d]\in H'\setminus H^p$. 

\vskip5mm
\begin{proof}[Proof of Theorem \ref{theoremA}] 
We have $G_1=H\times_M \langle z\rangle$ with $o(z)=p^2$ and $M=H\cap \langle z\rangle=\langle z^p\rangle$ a subgroup of order $p$ in 
$H$ but not in $H^p$. 

(1) Suppose $M=\langle [a,b]\rangle$. To show that $G_1$ is a Miller group, by Theorem \ref{theoremC}, it suffices to show that $H/M$ is a special Miller $p$-group. Thus $G_1=\langle a,b,c,d,z\rangle$.

Write $\overbar{H}=H/M=H/\langle [a,b]\rangle$. 
The following relations hold in $\overbar{H}=\langle \bar{a},\bar{b},\bar{c},\bar{d}\rangle$:
\begin{equation*}
1=[\bar{a},\bar{b}], \hskip3mm \bar{a}^p=[\bar{a},\bar{c}], 
\hskip3mm \bar{b}^p=[\bar{a},\bar{b}\bar{c}\bar{d}],
\hskip3mm \bar{c}^p=[\bar{b},\bar{c}\bar{d}], \hskip3mm \bar{d}^p=[\bar{b},\bar{d}].
\end{equation*}
It follows easily that $\overbar{H}$ is special $p$-group. It remains to show that $\overbar{H}$ is a Miller group, which 
can be proved almost by similar arguments as in the proof of Lemma \ref{lemma3}, hence we only sketch the proof.

The above relations in $\overbar{H}$ imply that 
$A=\langle \bar{a},\bar{b},Z(\overbar{H})\rangle$ is unique abelian subgroup of index $p^2$ in $\overbar{H}$, 
hence it is characteristic in $\overbar{H}$. 
Then $A^p=\langle \bar{a}^p,\bar{b}^p\rangle$ is characteristic in $\overbar{H}$. 
For $x\in \overbar{H}$ one can see that $[x,\overbar{H}]\leq A^p$ if and only if $x\in \langle a,Z(\overbar{H})\rangle$. Thus, 
$B=\langle a,Z(\overbar{H})\rangle$ is characteristic subgroup of $\overbar{H}$, and so is $B^p=\langle a^p\rangle$.

Consider $\widetilde{H}=\overbar{H}/B^p$. 
Then $\widetilde{H}=\langle \tilde{a},\tilde{b},\tilde{c},\tilde{d}\rangle$, 
and following relations hold in $\widetilde{H}$:
\begin{equation*}
[\tilde{a},\tilde{c}]=[\tilde{a},\tilde{b}]=1 \mbox{ and } \tilde{b}^p=[\tilde{a},\tilde{d}], \tilde{c}^p=[\tilde{b},\tilde{c}\tilde{d}], \tilde{d}^p=[\tilde{b},\tilde{d}].
\end{equation*}
From these relations, one can see that the elements of $\langle \tilde{a},\tilde{b},\tilde{c},Z(\overbar{H})\rangle$ 
have at most $p^2$ conjugates, while other elements have $p^3$ conjugates in $\widetilde{H}$. Hence 
$\langle \tilde{a},\tilde{b},\tilde{c},Z(\widetilde{H})\rangle$ is characteristic in $\widetilde{H}$. 
Consequently, $\langle \bar{a},\bar{b},\bar{c},Z(\overbar{H})\rangle$ is characteristic in $\overbar{H}$. Thus, 
$$Z(\overbar{H}) < \langle \bar{a},Z(\overbar{H})\rangle  < \langle \bar{a},\bar{b},Z(\overbar{H})\rangle < \langle \bar{a},\bar{b},\bar{c},Z(\overbar{H})\rangle < \overbar{H}$$ 
is a series of characteristic subgroups in $\overbar{H}$. Thus for any  $\varphi\in\mbox{Aut}(\overbar{H})$
\begin{align*}
\varphi(\bar{a})=\bar{a}^{i_1}z_1,   \hskip3mm \varphi(\bar{b}) = \bar{a}^{i_2}\bar{b}^{j_2}z_2,\hskip3mm \varphi(\bar{c}) = \bar{a}^{i_3}\bar{b}^{j_3}\bar{c}^{k_3}z_3 \hskip3mm \varphi(\bar{d}) = \bar{a}^{i_4}\bar{b}^{j_4}\bar{c}^{k_4}\bar{d}^{l_4}z_4,
\end{align*}
for some $z_1,z_2,z_3,z_4\in Z(\overbar{H})$.
Using the relations in $\overbar{H}$, one can deduce that $\varphi$ is identity on $\overbar{H}/Z(\overbar{H})$. Since $\overbar{H}$ is a special $p$-group, 
$\mathrm{Aut}(\overbar{H})$ is elementary abelian, i.e. $\overbar{H}$ is a Miller group.

(2) Suppose $M=H\cap \langle z\rangle =\langle [a,d]\rangle$. Since $M=\langle z^p\rangle$, we can assume, without loss of generality, that $z^p=[a,d]$. Thus $G_1=\langle a,b,c,d,z\rangle$ and the following relations hold in $G_1$:
\begin{align*}
& a^p=[a,c], \hskip3mm b^p=[a,bcd], \hskip3mm c^p=[b,cd], \hskip3mm d^p=[b,d],\\
& [z,a]=[z,b]=[z,c]=[z,d]=1, \hskip3mm z^p=[a,d].
\end{align*}

To show that $G_1$ is not a Miller group, consider the following map:
\begin{align*}
\varphi\colon a\mapsto az, \hskip3mm b\mapsto bz, \hskip3mm c\mapsto cd,\hskip3mm d\mapsto d,\hskip3mm z\mapsto z.
\end{align*}
It is easy to see that $\varphi$ respects all the relations in $G_1$. 
Clearly, $\varphi(c)\not\equiv c$ mod $Z(G_1)$. It follows that $\varphi$ is a non-central automorphism 
and $\mathrm{Aut}(G_1)$ is non-abelian.
\end{proof}
\vskip5mm
\begin{proof}[Proof of Theorem \ref{theoremB}]
We have following relations in $H=\langle a,b,c,d\rangle$, the $p$-group of class $2$:
\begin{equation}\label{eq2}
a^p=[a,c], \,\,\,\, b^p=[a,bcd], \,\,\,\, c^p=[b,cd], \,\,\,\, d^p=[b,d]. \tag{R1}
\end{equation}
It is easy to see that the map
$$\gamma\colon a\mapsto ad^p, \hskip3mm b\mapsto b, \hskip3mm c\mapsto c, \hskip3mm d\mapsto d$$
extends to a central automorphism of $H$, and it not inner since 
$$[a,H]=\langle [a,b],[a,c],[a,d] \rangle =\langle b^pa^{-p}[a,d]^{-1}, a^p, [a,d] \rangle=\langle a^p, b^p, [a,d]\rangle$$
and $d^p\notin [a,H]$; hence $a$ is not conjugate to $ad^p$. 

Define $G_2=H\rtimes _M \langle w\rangle$, where $M=H\cap \langle w\rangle$ is a subgroup of order in $H'$ but not in $H^p$, 
$o(w)=p^3$ and $w$ normalizes $H$ in the following way:
\begin{equation}\label{eq3}
 waw^{-1}=ad^p,\hskip3mm wbw^{-1}=b, \hskip3mm wcw^{-1}=c, \hskip3mm wdw^{-1}=d.\tag{R2}
\end{equation}
Then $G_2$ is $p$-group of order $p^{12}$, and 
$$\Phi(G_2)=Z(G_2)=\langle Z(H),z^p\rangle, \hskip1cm G_2'=H'.$$ 

(1) Suppose $M=H\cap \langle w\rangle = \langle [a,b]\rangle$. From the proof of Theorem \ref{theoremA}, $H/M$ is a special Miller $p$-group. 
Therefore, by Theorem \ref{theoremD}, $G_2$ is Miller group.

(2) Suppose that $M=H\cap \langle w\rangle = \langle [a,d]\rangle$. Since $|H\cap \langle w\rangle|=p$ and $o(w)=p^3$, we have 
$\langle [a,d]\rangle=H\cap \langle w\rangle=\langle w^{p^2}\rangle$. Without loss of generality, we can assume that 
\begin{equation}\label{eq4}
w^{p^2}=[a,d].\tag{R3}
\end{equation}
Then the elements $a,b,c,d,w$ with relations \ref{eq2}, \ref{eq3}, \ref{eq4} give a presentation of $G_2$, and it is easy to see that the following map respects all the relations in $G_2$:
\begin{equation*}
\varphi\colon a\mapsto aw^p,  \hskip3mm b\mapsto bw^p, \hskip3mm c\mapsto cd, \hskip3mm d\mapsto d, \hskip3mm w\mapsto w.
\end{equation*}
It follows that $\varphi$ extends to an automorphism of $G_2$, which is clearly 
non-central, and $G_2$ is not a Miller group. 
\end{proof}

We conclude our paper by some analysis on certain special Miller $p$-groups $H$, in which, it turns out that for most of the subgroups $M$ of order $p$ in $H'$ but not in $H^p$, $H/M$ is a special Miller $p$-group. Consider the following $p$-groups $H_1$ and $H_2$ of class $2$, for which we take same set of generators $\{a,b,c,d\}$ but they have the following relations respectively:
\begin{align*}
R_1: \hskip3mm & a^p=[a,c], \hskip3mm  b^p=[a,bcd], \hskip3mm  c^p=[b,cd], \hskip3mm  d^p=[b,d], \hskip3mm [c,d]=1
\end{align*}
and
\begin{align*}
R_2: \hskip3mm & a^p=[a,c], \hskip3mm b^p=[a,cd], \hskip3mm  c^p=[b,cd], \hskip3mm  d^p=[b,d].
\end{align*}
Note that $H_1$ and $H_2$ are special Miller $p$-groups of order $p^9$ and $p^{10}$ respectively.
Also note that $H_1$ and $H_2$, respectively, satisfy conditions {\bf (I)}-{\bf (II)} and {\bf (I)}-{\bf (III)} as given in \S \ref{section3}.
The following table describes our analysis, in which $M$ denotes a subgroup of order $p$ in $H_i'$ but not in $H_i^p$.
\begin{center}
\begin{tabular}{| c | c | c | c |}
\hline
 Group $H_i$ & Prime & Number of subgroups $M$          &$|\{ M\colon H_i/M$ is special Miller $\}|$ \\ \hline
                    & 3        &       81                                      &        76     \\ \cline{2-4}
 $H_1$         & 5        &   625                                        &    616         \\ \cline{2-4}
                    & 7        &             2401                             &    2388         \\ \hline
                    & 3        &        324                                    &         318    \\ \cline{2-4}
 $H_2$         & 5        &     3750                                      &      3740       \\ \cline{2-4}
                    & 7        &         19208                                &     19197        \\ \hline
\end{tabular}
\end{center}
We remark that the values in the above table have been computed using GAP \cite{gap} and Magma \cite{bosma}.

By the above analysis, the following question arises. 

\noindent{\bf Question:} {\it Let $H$ be a special Miller $p$-group satisfying one of the following conditions:

(1) $|H/H'|=|H^p|$ and $H^p<H'$.

(2) $|H/H'|=|H^p|$, $H^p<H'$ and $H'$ is freely generated by $\{ [x_i,x_j]\colon 1\leq i<j\leq n\}$ provided $\{x_1,\ldots,x_n\}$ is a minimal generating set for $H$.

Does there exist a subgroup $M$ of order $p$ in $H'$ but not in $H^p$ such that $H/M$ is  a special Miller $p$-group?
}

If this question has an affirmative answer, then we can always construct a non-special Miller $p$-group from a special Miller $p$-group $H$ satisfying condition (1) or (2) by Carant's methods, as illustrated in Theorems \ref{theoremC} and \ref{theoremD}.
%
%

%
%
\vskip2mm\noindent
\scriptsize
Rahul Dattatraya Kitture\\
Post-Doctoral Fellow,\\
School of Mathematics,\\ Harish-Chandra Research Institute,\\
 Chhatnag Road, Jhunsi, Allahabad - 211 019, INDIA.\\
 Email: {\tt rahul.kitture@gmail.com}

\vskip2mm\noindent
Manoj K. Yadav\\
School of Mathematics,\\
 Harish-Chandra Research Institute,\\
 Chhatnag Road, Jhunsi, Allahabad - 211 019, INDIA.\\ Email: {\tt myadav@hri.res.in}
\end{document}